\documentclass[11pt]{amsart}
\usepackage{latexsym, mathrsfs, color, tikz, multirow, bbm, mathtools,amsmath,amssymb}
\usepackage{graphics}
\usepackage{subcaption}
\input{xy}
\xyoption{all}

\usepackage[colorlinks=true, pdfstartview=FitV,
 linkcolor=black,citecolor=black,urlcolor=black]{hyperref}
 
\usepackage{tikz}
\usepackage{tikz-cd}
\usepackage{comment}
\usepackage{adjustbox}

\numberwithin{equation}{section}
\theoremstyle{plain}

\newtheorem{Prop}[equation]{Proposition}
\newtheorem{Cor}[equation]{Corollary}
\newtheorem{Lem}[equation]{Lemma}

\newtheorem{Open}[equation]{Open Problem}

\theoremstyle{definition}
\newtheorem{Def}[equation]{Definition}
\newtheorem{Exa}[equation]{Example}
\newtheorem{Rmk}[equation]{Remark}

\newenvironment{red}{\relax\color{red}}{\relax}
\newenvironment{blue}{\relax\color{blue}}{\hspace*{.5ex}\relax}

\newcommand{\ber}{\begin{red}}
\newcommand{\er}{\end{red}}
\newcommand{\beb}{\begin{blue}}
\newcommand{\eb}{\end{blue}}

\newcommand{\M}{\mathcal{M}}
\newcommand{\MA}{\mathcal{M_A}}
\newcommand{\MRed}{\mathcal{M_R}}

\begin{document}
	
\title[Mutation class poset and topology]{A topology on the poset of quiver mutation classes}

\author[T. J. Ervin]{Tucker J. Ervin}
\address{Department of Mathematics, University of Alabama,
	Tuscaloosa, AL 35487, U.S.A.}
\email{tjervin@crimson.ua.edu}
\author[B. Jackson]{Blake Jackson}
\address{Department of Mathematics, University of Connecticut,
	Storrs, CT 06269, U.S.A.}
\email{blake.jackson@uconn.edu}

\begin{abstract}
    To better understand mutation-invariant and hereditary properties of quivers (and more generally skew-symmetrizable matrices), we have constructed a topology on the set of all mutation classes of quivers which we call the mutation class topology. 
    This topology is the Alexandrov topology induced by the poset structure on the set of mutation classes of quivers from the partial order of quiver embedding.
    The closed sets of our topology---equivalently, the lower sets of the poset---are in bijective correspondence with mutation-invariant and hereditary properties of quivers.  
    We show that this space is strictly $T_0$, connected, non-Noetherian, and that every open set is dense.
    We close by providing open questions from cluster algebra theory in the setting of the mutation class topology and some directions for future research. 
\end{abstract}

\maketitle

\section{Introduction} \label{sec-intro}

Since their introduction in 2001 \cite{fomin_cluster_2001}, cluster algebras have been widely studied, including their algebraic structure \cite{fomin_cluster_2003,fomin_cluster_2007,gross_canonical_2018,lin_two_2024,muller_locally_2013,nakanishi_tropical_2012}, their quiver representations \cite{assem_elements_2006,buan_tilting_2006,caldero_quivers_2006,caldero_triangulated_2008,derksen_quivers_2008,gross_canonical_2018,schiffler_quiver_2014,speyer_acyclic_2013}, and their combinatorics \cite{beineke_cluster-cyclic_2011,cao_uniform_2019,najera_chavez_c-vectors_2012,fomin_long_2023,fomin_cluster_2008,he_geometric_2023,seven_cluster_2015,seven_cluster_2019}.
Recent advances in cluster algebra theory have attempted to understand \textit{mutation-invariant and hereditary properties} of quivers \cite{ervin_new_2024,fomin_universal_2021}---properties which are invariant under both mutation and restriction to any full subquiver. 
In the process of furthering this line of research, combining existing mutation-invariant and hereditary properties to generate new properties led us to a topology on the space of mutation classes of quivers. 

We began with observing hereditary properties through quiver embedding, which resulted in a large poset whose elements consist of mutation classes of quivers. 
In this setting, a property is mutation-invariant and hereditary if and only if it is a lower set (or down set) of the poset.
Discussing the intersections and unions of lower sets was a natural next step, leading to the construction of a topology on our poset.
The topology was the Alexandrov topology generated by a specialization (or canonical) preorder on the set of all mutation classes of quivers \cite{alexandroff_diskrete_1937,arenas_alexandroff_1999}.
The preorder is induced by quiver embedding\footnote{It can also be induced by quiver restriction \cite[Definition 2.4]{fomin_universal_2021}---which is a dual notion to embedding---and we use the terms interchangeably throughout the paper when convenient.} on the mutation classes. 
The importance of defining our topology via quiver embedding comes from the lower sets of our poset: the closed sets of this topology are, by definition, the lower sets.
This means that the set of all quivers sharing a particular mutation-invariant and hereditary property form a closed set in the topology and vice versa.
Therefore, under this topology, the study of mutation-invariant and hereditary properties is the same as the study of the closed sets of the topology, making this the most natural topology for observing mutation-invariant and hereditary properties.
All of the preliminaries for this construction can be found in Section~\ref{sec-pre}, and we define and give some basic results of the mutation class topology in Section~\ref{sec-mut}.

We also naturally extend this topology to the space of mutation classes of skew-symmetrizable $(n+m) \times (n+m)$ integral matrices.
The skew-symmetrizable space contains the quiver space as a subspace, and so the topology defined in Section~\ref{sec-mut} is the subspace topology of the Alexandrov topology on the much larger space of mutation classes of skew-symmetrizable matrices.
Therefore, we can view the poset and topology described in this paper as a poset and topology on the space of all cluster algebras of geometric type\footnote{The set of skew-symmetrizable $(n+m) \times (n+m)$ integral matrices contains the set of \textit{cluster patterns of geometric type} as a subset, and we can define the cluster algebras of geometric type from the cluster patterns of geometric type \cite{nakanishi_cluster_2023}.}.
All of the properties outlined in Section~\ref{sec-mut} can also be extended to the larger space.

At the end of Section~\ref{sec-mut}, we show that there exists a nontrivial bi-infinite chain of closed sets in the mutation class topology, meaning that our space satisfies neither the ascending or descending chain condition on closed sets.
This forces the space to be non-Noetherian.
We also show that every open set in our space is a dense set; in other words, if $S$ is an open set, the closure of $S$ is the entire space.
In Section~\ref{sec-open}, we list a few open problems stated in terms of the mutation class topology that we define in Section~\ref{sec-mut}.
Some of these questions are related to existing open questions in cluster algebra theory. 
For example, the property of a quiver being mutation-acyclic is a mutation-invariant and hereditary property \cite[Corollary 5.3]{buan_cluster_2008}.
Some have explored the combinatorial reason for this phenomenon, but to the authors' knowledge, no one has given a combinatorial proof of this result\footnote{This pursuit is what sparked our current research.}. 
Since this is a mutation-invariant and hereditary property, we can view this problem in terms of the closed sets of our topology and perhaps find a simpler approach.
Other questions in Section~\ref{sec-open} involve Banff and Louise quivers.
It is an open question as to whether or not the two sets equal one another \cite{bucher_banff_2021,ervin_answering_2023}, and Open Problems~\ref{open-closure} and~\ref{open-banff-louise} give a new approach to answer this open question through the lens of our mutation class topology and the closure operation.
As such, we believe that the study of quivers (and more generally cluster algebras of geometric type) from the point of view of the poset and/or topology of mutation classes should be an avenue that researchers consider when asking questions involving mutation classes.

\section{Preliminaries} \label{sec-pre}

For completeness, we begin with the basic definitions of quivers and posets. Throughout the paper, we use $[n] = \{1,2,\dots,n\}.$

\begin{Def} \label{def-quiver-mutation}
    A \textit{quiver} $Q$ is a finite multidigraph without loops and oriented 2-cycles.
    The vertices of a quiver are labeled with the numbers $\{1,2,\dots,n\}$, and the directed edges of $Q$ are called \textit{arrows}.
    If $k \in [n]$ is a vertex of $Q$, then the \textit{mutation of $Q$ at k} is the quiver $\mu_k(Q)$ obtained from $Q$ in the following way:
    \begin{enumerate}
        \item for each oriented 2-path $i \to k \to j$ in $Q$, add an arrow $i \to j$ in $\mu_k(Q)$,
        \item reverse the direction of all arrows incident to $k$,
        \item pairwise delete any arrows which form an oriented 2-cycle.
    \end{enumerate}
    A quiver $P$ is \textit{mutation-equivalent} to a quiver $Q$ if there is a finite sequence of vertices $[i_1, i_2, \dots, i_\ell]$ such that $P$ is isomorphic to $\mu_{i_\ell}(\mu_{i_{\ell-1}}( \dots \mu_{i_1}(Q) \dots ) )$.
    The \textit{mutation class} $[Q]$ of $Q$ is the collection of all quivers $P$ that are mutation-equivalent to $Q$.
\end{Def}

\begin{Def} \label{def-subquiver}
    Let $Q$ be a quiver on $n$ vertices and let $I \subseteq \{1,2,\dots,n\}$.
    Then the \textit{full subquiver on I} is the quiver $Q_I$ with vertices $I$ and arrows $\{i \to j \in Q \ | \ i,j \in I\}$.
    This is sometimes referred to as the \textit{restriction of $Q$ to $I$} \cite{fomin_long_2023, fomin_introduction_2021}.
\end{Def}

\begin{Def} \label{def-poset}
    A \textit{partially ordered set} or \textit{poset} is a set $X$ together with a relation $\preceq$ such that, for any $x,y,z \in X$,
    \begin{enumerate}
        \item $x \preceq x$ (reflexivity);
        \item if $x \preceq y$ and $y \preceq x$, then $x = y$ (antisymmetry); and
        \item if $x \preceq y$ and $y \preceq z$, then $x \preceq z$ (transitivity).
    \end{enumerate}
\end{Def}

Now we illustrate a poset structure on the set of all mutation classes of quivers induced by \textit{embedding}.

\begin{Def}[\textup{\cite[Definition 4.1.8]{fomin_universal_2021}}] \label{def-embedding}
    A mutation class $[P]$ \textit{embeds} into a mutation class $[Q]$ whenever there exists a $P' \in [P]$ that is isomorphic to a full subquiver of a quiver $Q' \in [Q]$.
    Equivalently, the mutation class $[Q]$ \textit{restricts} to $[P]$.
    Note that the relation of embedding is reflexive and transitive.
    Additionally, since our mutation-equivalence considers isomorphic quivers to be equivalent, it is also antisymmetric.
    We denote the partial ordering produced on the set of all mutation classes by $\preceq$, where $[P] \preceq [Q]$ whenever $[P]$ embeds into $[Q]$.
    We may drop the brackets representing the mutation class and refer to a quiver $P$ embedding into $Q$ or a quiver $Q$ restricting to $P$.
\end{Def}

By only considering mutation classes of quivers together with the relation $\preceq$, we can build the mutation class poset.
A small snippet of this poset is presented in Figure~\ref{fig-poset}. 
There are a few interesting things to note about the mutation class poset. 
The first thing to note is that it has a unique minimum element or a zero: the trivial quiver with 1 vertex and no arrows.
This quiver is the sole member of its mutation class, and it embeds into any quiver with at least one vertex. 
This fact will be useful when we prove our topological space is connected.
This poset is unbounded; in other words, given any mutation class $[Q]$, we can find another mutation class $[\widehat{Q}]$ into which $[Q]$ properly embeds.
It also has a well-defined rank function, which simply returns the number of vertices of a quiver in a mutation class.
For every rank $n$ of the poset (except the $1^{st}$), it is not too hard to see that there are countably many mutation classes with rank $n$. 
This is because the number of mutation classes of rank $n$ has the same cardinality as the set of $n \times n$ integral skew-symmetric matrices up to mutation equivalence.
Finally, this poset is not a lattice since the join and meet operations are not well defined.
It may be interesting to study this poset or finite subposets in its own right.
For example, observing rowmotion \cite{panyushev_orbits_2009,propp_homomesy_2015,striker_promotion_2012} may be worthwhile since it acts on the order ideals of a poset.

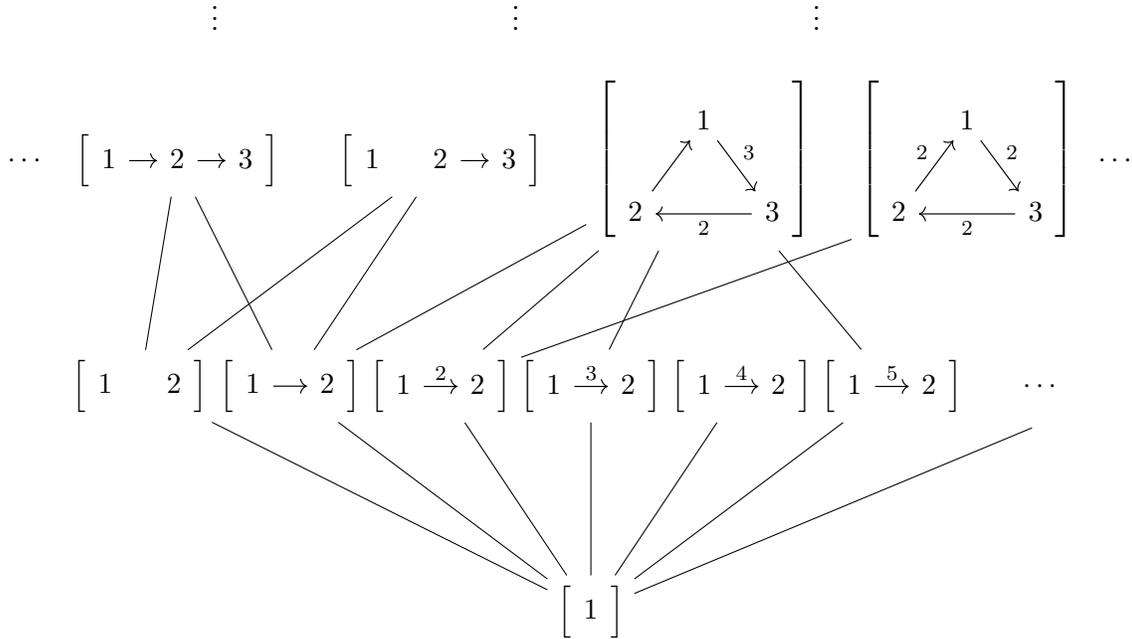
\begin{figure}
    \centering
    \begin{adjustbox}{max width=\textwidth}
        \begin{tikzpicture}
            \node (31) at (-4,5) {$\vdots$};
            \node (32) at (0,5) {$\vdots$};
            \node (33) at (4,5) {$\vdots$};
          \node (21) at (-6.5,3) {$\cdots$}; 
          \node (22) at (-4.5,3) {$\left[\begin{tikzcd}[column sep = 1em]
                                    1 & 2 & 3
                                    \arrow[from=1-1, to=1-2]
                                    \arrow[from=1-2, to=1-3]
                                \end{tikzcd}\right]$}; 
        \node (225) at (-1,3) {$\left[\begin{tikzcd}[column sep = 1em]
                                    1 & 2 & 3
                                    \arrow[from=1-2, to=1-3]
                                \end{tikzcd}\right]$}; 
          \node (23) at (2.5,3) {$\left[\begin{tikzcd}[column sep = 1em]
                                    & 1 & \\
                                    2 & & 3
                                    \arrow[from=1-2, to=2-3, "3"]
                                    \arrow[from=2-3, to=2-1, "2"]
                                    \arrow[from=2-1, to=1-2]
                                \end{tikzcd}\right]$};
        \node (24) at (6,3) {$\left[\begin{tikzcd}[column sep = 1em]
                                    & 1 & \\
                                    2 & & 3
                                    \arrow[from=1-2, to=2-3, "2"]
                                    \arrow[from=2-3, to=2-1, "2"]
                                    \arrow[from=2-1, to=1-2, "2"]
                                \end{tikzcd}\right]$};
        \node (245) at (4.6,2) {};
        \node (25) at (8,3) {$\cdots$};
          \node (11) at (-5,0) {$\left[\begin{tikzcd}[column sep = 1em]
                                     1 & 2 
                                \end{tikzcd}\right]$};
          \node (12) at (-3,0) {$\left[\begin{tikzcd}[column sep = 1.2em]
                                    1 & 2 
                                    \arrow[from=1-1, to=1-2]
                                \end{tikzcd}\right]$};
          \node (13) at (-1,0) {$\left[\begin{tikzcd}[column sep = 1.2em]
                                    1 & 2 
                                    \arrow[from=1-1, to=1-2, "2"]
                                \end{tikzcd}\right]$};
          \node (14) at (1,0) {$\left[\begin{tikzcd}[column sep = 1.2em]
                                    1 & 2 
                                    \arrow[from=1-1, to=1-2, "3"]
                                \end{tikzcd}\right]$};
          \node (15) at (3,0) {$\left[\begin{tikzcd}[column sep = 1.2em]
                                    1 & 2 
                                    \arrow[from=1-1, to=1-2, "4"]
                                \end{tikzcd}\right]$};
          \node (16) at (5,0) {$\left[\begin{tikzcd}[column sep = 1.2em]
                                    1 & 2 
                                    \arrow[from=1-1, to=1-2, "5"]
                                \end{tikzcd}\right]$};
          \node (17) at (7,0) {$\cdots$};
          \node (18) at (7,-0.5) {};
          \node (zero) at (1,-3) {$\left[\begin{tikzcd}
                                    1
                                \end{tikzcd}\right]$};
          \draw (zero) -- (11) -- (22) -- (12) -- (zero) -- (13) -- (23) -- (14) -- (zero) -- (15);
          \draw (13) -- (245);
          \draw (zero) -- (16) -- (23);
          \draw (zero) -- (18);
          \draw (12) -- (23);
          \draw (11) -- (225) -- (12);
        \end{tikzpicture}
    \end{adjustbox}
    \caption{A snippet of the quiver mutation class poset}
    \label{fig-poset}
\end{figure}

Next, we define mutation-invariant and hereditary properties.  
These quiver properties are, in short, inherited by the ``children'' of the quiver.
Moreover, the mutation-invariance makes these properties of not just quivers, but mutation classes of quivers.
While the research on some of these properties goes back quite a long way (for example, the existence of a reddening sequence is mutation-invariant and hereditary), there are only a few currently of interest.
We hope that the current paper will highlight the importance of considering multiple mutation-invariant and hereditary properties simultaneously, and not in isolation. 
For a more formal treatment of these properties, we direct the reader to \cite{ervin_new_2024,fomin_universal_2021}.

\begin{Def} \cite[Defintion 4.1.3]{fomin_introduction_2021}
    We say that a property $\mathcal{F}$ on quivers is a \textit{mutation-invariant and hereditary property} if, for all $Q$ having property $\mathcal{F}$,
    \begin{enumerate}
        \item $P \in [Q]$ has property $\mathcal{F}$ and
        \item $P' \in [P_I]$ has property $\mathcal{F}$ for all subquivers $P_I$ of $P$.
    \end{enumerate}
    Therefore, a mutation-invariant and hereditary property $\mathcal{F}$ is preserved under the embedding operation of Definition~\ref{def-embedding}.
    In other words, if $[Q]$ has a property $\mathcal{F}$ and $[P] \preceq [Q]$, then $[P]$ has property $\mathcal{F}$.
\end{Def}

\begin{Exa}
    The easiest mutation-invariant and hereditary property to describe is the property of embedding into $[Q]$ for some quiver $Q$. 
    In Figure~\ref{fig-poset}, we can see the $A_3$ quiver in the upper-left-hand portion of the figure. 
    The property of embedding into $[A_3]$ is a mutation-invariant and hereditary property, and the set of mutation classes with this mutation-invariant and hereditary property is
    $$S = \{[A_3], [A_2], [1], [\begin{tikzcd}[column sep = 1em]
                                     1 & 2 
                                \end{tikzcd}]\}$$
\end{Exa}

\begin{Exa} 
    We denote the set of all mutation-acyclic quivers as $\MA$ and the set of all quivers with a reddening sequence as $\MRed$.
    Both of these are mutation-invariant and hereditary properties \cite{buan_cluster_2008, muller_existence_2016}, and we will reference them throughout the paper.
\end{Exa}

For completeness, we also include the definition of a topology and the Alexandrov topology. 

\begin{Def}[\cite{munkres_topology_2019}] \label{def-topology}
    A \textit{topology} on a set $X$ is a collection $\mathcal{T}$ of subsets of $X$ satisfying the following conditions:
    \begin{enumerate}
        \item Both the empty set and $X$ belong to $\mathcal{T}$;

        \item Arbitrary intersections of sets in $\mathcal{T}$ belong to $\mathcal{T}$;

        \item Finite unions of sets in $\mathcal{T}$ belong to $\mathcal{T}$.
    \end{enumerate}
    The sets in $\mathcal{T}$ are said to be the \textit{closed sets}, and the pair $(X,\mathcal{T})$ is called a \textit{topological space}.
\end{Def}

\begin{Rmk}
    The usual definition of a topology on $X$ is formulated in terms of the open sets. 
    However, our focus is on closed sets, so we use the alternative definition.
\end{Rmk}

\begin{Def}[\cite{alexandroff_diskrete_1937,arenas_alexandroff_1999}] \label{def-alexandrov}
    Given a set $X$ with a partial order $\preceq$, we may form a topology $\mathcal{T}$ on $X$ by taking the set of closed sets to be
    $$\{ S \subseteq X \ | \ \forall x,y \in X, \text{ if } x \in S\text{ and }y \preceq x \implies y \in S \}. $$
    Moreover, the open sets of $X$ are given by
    \[ \{ U \subseteq X \ | \ \forall x,y \in X, \text{ if } y \in U \text{ and } y \preceq x \implies x \in U \}. \]
    The closed sets $S$ are the lower sets or the down sets of the poset, and the open sets $U$ are upper sets or upsets of the poset.
    The resulting topology $\mathcal{T}$ is called the \textit{Alexandrov topology} and $(X,\mathcal{T})$ is called an \textit{Alexandrov-discrete space} (sometimes just an \textit{Alexandrov space}).
\end{Def}

\section{The mutation class topology} \label{sec-mut}

\begin{Def} \label{def-mutation-topology}
    Let $\M$ be the set of all mutation classes of quivers up to isomorphism.
    The \textit{mutation class topology} on $\M$ is the Alexandrov topology induced by $\preceq$. In other words, the closed sets of $\M$ are the lower sets of the poset described in Section~\ref{sec-pre}: the sets $S$ such that $[Q] \in S$ implies that $[P] \in S$ for all $[P] \preceq [Q]$.
\end{Def}

Again, the benefit of constructing this topology on $\M$ is that the closed sets are in bijective correspondence with mutation-invariant and hereditary properties of quivers.
Any property that can be had by mutation classes is naturally a mutation-invariant property of quivers.
The hereditary properties of quivers are passed to their full subquivers.
Hence, mutation-invariant and hereditary properties are exactly the properties that are preserved by embedding into a mutation class with that property. 
This forces the mutation-invariant and hereditary properties of quivers to correspond exactly with closed sets in $\M$.

\subsection{Properties arising from Alexandrov topologies} \label{subsec-alexandrov}

We start this subsection with basic results about the space; namely the separation axioms it satisfies. 

\begin{Def}[\cite{schechter_handbook_1996}] \label{def-separation}
    A topological space $X$ is $T_0$ or \textit{Kolmogorov} if any two points of $X$ are topologically distinguishable, i.e., for any two distinct points there is a closed set containing one and not the other. A space is $R_0$ or \textit{symmetric} if any two topologically distinguishable points in $X$ are separated; in other words, the specialization preorder on the points of $X$ is a symmetric relation (an equivalence relation). A space is $T_1$ if it is both $T_0$ and $R_0$. 
\end{Def}

The first thing to note is that every Alexandrov-discrete space is $T_0$, and some of them are also $T_1$.
Therefore, we can be assured that $\M$ is $T_0$.
On the other hand, $\M$ is not $R_0$ since our specialization preorder is not symmetric \cite[Example 16.8.e]{schechter_handbook_1996} (in fact, it is antisymmetric), therefore it cannot be $R_0$ or $T_1$. 
Since our space is not $R_0$ or $T_1$, it fails to meet any other separation axioms---for example, regular spaces must satisfy $R_1$ (and therefore $R_0$) and Hausdorff and normal spaces must satisfy $T_1$.

We were able to find a few non-standard properties.
It is known that Alexandrov-discrete spaces are locally path-connected~\cite{arenas_alexandroff_1999}, but the connectedness of $\M$ is not guaranteed. 

\begin{Prop} \label{prop-clopen}
    The only subsets of $\M$ that are both open and closed are the empty set and $\M$.
    Thus, $\M$ is connected.
\end{Prop}

\begin{proof}
    Assume we have a non-empty clopen set $V \subseteq \M$.
    Since $V$ is closed, it must be a down set in the poset, and therefore it must contain the mutation class of the trivial quiver, as the trivial quiver embeds into every mutation class.
    Since $V$ is open, its complement $U$ in $\M$ is also closed, meaning $U$ either contains the mutation class of the trivial quiver or is empty.
    Since $U \cap V = \emptyset$, we must have that $U = \emptyset$.
    As such, we have $V = \M$.
\end{proof}

Additionally, we can show that $\M$ is a (trivially) compact space, again using the idea that our poset contains a minimal element.

\begin{Def}
    An \textit{open cover} of a space $X$ is a collection of open sets $\{U_i\}_{i \in \mathcal{I}}$ such that \[ \bigcup_{i \in \mathcal{I}}U_i = X. \] A topological space $X$ is \textit{compact} if every open cover of $X$ admits a finite subcover; in other words, $\mathcal{I}$ can be taken to be a finite set.
\end{Def}

\begin{Prop} \label{prop-compact}
    $\M$ is compact.
\end{Prop}

\begin{proof}
    Let $\{U_i\}_{i \in \mathcal{I}}$ be any collection of open sets such that $\M = \bigcup_{i \in \mathcal{I}} U_i$.
    This implies that the mutation class of the trivial quiver is contained in some $U_i$.
    As such, the complement of $U_i$ is a closed set that does not contain the mutation class of the trivial quiver, making it an empty set.
    Thus $U_i = \M$ for some $i \in \mathcal{I}$.
    This means that any open cover of $\M$ can be reduced to a finite subcover, forcing $\M$ to be compact.
\end{proof}

We also prove that every open subset of $\M$ is dense, by first noting that every pair of mutation classes embeds into at least one quiver.

\begin{Def}
    If $(X,\mathcal{T})$ is a topological space and $A \subseteq X$ is a subset of $X$, then the \textit{closure} of $A$, denoted by $\overline{A}$, is the intersection of all the closed sets of $X$ containing $A$. Moreover, $A$ is a \textit{dense} subset of $X$ if $\overline{A} = X$.
\end{Def}

\begin{Rmk} \label{rmk-closure}
    In the mutation class topology, $\overline{A}$ is the lower set generated by $A$ in $\M$. Specifically, \[\overline{A} = \{ [Q] \in \M \ | \ [Q] \preceq [P] \text{ for some } [P] \in A \}. \]
\end{Rmk}

\begin{Prop} \label{prop-open-dense}
    Every nonempty open subset of $\M$ is dense.
\end{Prop}

\begin{proof}
    Let $U$ be any nonempty open subset of $\M$ and $[P]$ be any mutation class not contained in $U$.
    For any $[Q] \in U$, we may form the mutation class of the disjoint union of $[P]$ and $[Q]$: $[P] \sqcup [Q] = [P \sqcup Q]$.
    Note that this is a well-defined product of mutation classes since mutations in one connected component of a quiver do not affect other connected components.
    Both $[P]$ and $[Q]$ naturally embed into $[P \sqcup Q]$, and we know that $[P \sqcup Q] \in U$ since $U$ is an upper set of our poset.
    The closure of $U$ then includes the quiver $[P]$ by Remark~\ref{rmk-closure}.
    As $[P]$ was arbitrary, this demonstrates that the closure of $U$ is the whole space $\M$, making $U$ dense. 
\end{proof}

Finally, since the poset is ranked by the number of vertices in the quiver, there are no finite dense subsets of $\M$.
The closure of any such subset would not include any quivers with sufficiently large rank.

\begin{Cor} \label{cor-finite-dense}
    Every dense subset of $\M$ is an infinite set.
\end{Cor}

There are also properties of Alexandrov-discrete spaces that are known which we restate for our space.
Arenas \cite{arenas_alexandroff_1999} showed that subspaces of Alexandrov-discrete spaces are again Alexandrov-discrete.
This means that if we want to consider, say, quivers with at most $n$ vertices, $\MA$, or $\MRed$ as spaces of mutation classes, then these spaces are again Alexandrov-discrete under the induced subspace topology; therefore, we will again have mutation-invariant and hereditary properties corresponding to closed sets.
This property will prove itself useful again in Section~\ref{sec-iced}, where we show that $\M$ is a subspace of a more general mutation class space.
The second property is that quotient spaces of Alexandrov-discrete spaces are Alexandrov-discrete spaces. 
Therefore, if we were to introduce an equivalence relation $\sim$ on $\M$, then $\M/{\sim}$ would again be Alexandrov-discrete. 
The third property is that finite products of Alexandrov-discrete spaces are Alexandrov-discrete.
This means one could study the space $\M^n$ consisting of points which are $n$-tuples of mutation classes $([Q_1], [Q_2], \dots, [Q_n])$ and still enjoy the same results that we have in this paper (with the proper adjustments to the statements).
This opens the possibility to study some interesting continuous functions, such as the product map. If we have a well-defined notion of ``products'' of mutation classes of quivers, say $[Q] \times [P]$, then this operation is a continuous map from $\M^2 \to \M$.
Since both spaces are Alexandrov-discrete spaces, there might be some interesting properties of this map.
We don't have any examples of a well-defined directed graph product that respects mutation classes aside from the disjoint union of two mutation classes.

\subsection{Applications involving mutation-invariant and hereditary properties} \label{subsec-applications}

As one might expect, the mutation-finite quivers are exceptions to the usual closed sets in the mutation class topology.

\begin{Lem} \label{lem-finite-hereditary}
    Let $S \neq \emptyset$ be any finite closed subset of $\M$.
    Then every $[P] \in S$ is mutation-finite, i.e., there are only finitely many quivers in $[P]$.
\end{Lem}

\begin{proof}
    Suppose that there exists a $[P] \in S$ that is not mutation-finite.
    Then an infinite number of quivers on two vertices embed into $[P]$.
    As $S$ is closed, this would force $S$ to be an infinite set, a contradiction of our assumption.
    Hence, every $[P] \in S$ is mutation-finite.
\end{proof}

Additionally, as we are working with closed sets, we can discuss mutation-invariant and hereditary properties in terms of their open complements.

\begin{Def} \label{def-quiver-avoiding}
    Let $P$ be a quiver.
    We say that a quiver $Q$ or mutation class $[Q]$ is $[P]$-avoiding whenever $[P]$ does not embed into $[Q]$.
    Similarly, if $S$ is a set of quivers, we say that a quiver $Q$ is $S$-avoiding whenever $Q$ is $[P]$-avoiding for every $[P] \in S$.
\end{Def}

\begin{Lem} \label{lem-quiver-avoiding}
    Let $N_S$ be the set of all $S$-avoiding mutation classes and $O_S$ be its complement.
    Then $O_S$ is the open set generated by $S$, and $N_S$ is closed.
\end{Lem}

\begin{proof}
    Let $Q$ be some quiver such that $[P] \preceq [Q]$ for any $[P] \in S$.
    Then $Q$ is not $S$-avoiding, forcing $[Q] \in O_S$.
    This means that $O_S$ is the upper set generated by the elements of $S$.
    As upper sets are open in $\M$, we have our desired result. 
\end{proof}

This has already been used informally.
For example, we can partition the set of all mutation classes into mutation-acyclic classes, $\MA$, and non-mutation-acyclic classes.
This has been explored most thoroughly in quivers of rank 3 with regards to the Markov constant \cite{beineke_cluster-cyclic_2011}.
Similarly, the mutation-finite quivers form a closed subset of $\M$, and the corresponding complement of mutation-infinite quivers is open.
This partitioning can be done for any mutation-invariant and hereditary property, and its open complement $O_S$ has already been partially explored with universal collections.

\begin{Def}[\cite{fomin_universal_2021}] \label{def-universal-col}
    A \textit{universal collection} is a proper subset of $\M$ such that every quiver embeds into some quiver of the set.
    Hence, a subset of $\M$ is dense if and only if it is a universal collection.
\end{Def}

\begin{Def} \cite[Defintion 2.6]{fomin_universal_2021}
    Let $n$ be an integer greater than 1.
    Then a quiver $Q$ is $n$-universal if every quiver on $n$ or less vertices embeds into $Q$.
\end{Def}

The first example of a universal collection is any collection of $n$-universal quivers where every $n > 1$ is represented \cite[Remark 5.6]{fomin_universal_2021}.
As discussed previously, any mutation-invariant and hereditary property gives us a corresponding open subset---a dense set and a universal collection.
Thus, open sets in the mutation class topology are universal collections.

\begin{Cor} \label{cor-universal-collections}
    Let $\mathcal{H}$ be any mutation-invariant and hereditary property.
    Then $\mathcal{H} = N_S$ for at least one subset $S \subseteq \M$, where $O_S$ is the corresponding universal collection that is also an open set.
    Additionally, if $S$ is any universal collection (open or otherwise), then $N_S$ is the corresponding mutation-invariant and hereditary property and $O_S$ is the open set generated by $S$.
\end{Cor}

We end by demonstrating how translating between closed sets and mutation-invariant and hereditary properties can lead to new examples of both.

\begin{Rmk} \label{rmk-mutation-N-abundant}
    Let $C_N$ be the set of all mutation classes of quivers that are mutation $N$-abundant, i.e. there are $N$ or more arrows between any pair of vertices in any quiver of the mutation class.
    When you note that the single vertex quiver is trivially mutation $N$-abundant for all $N$, we have a natural mutation-invariant and hereditary property.
    The sets $C_N$ are then closed for all $N$.
    Furthermore, it can be easily seen that 
    $$\cdots C_3 \subset C_2 \subset C_1$$
    with $C_{N} \neq C_{N+1}$.
    If we re-imagine $C_N$ as a set of quivers avoiding some set $S$, then we quickly find that $S$ is the set of all rank two quivers with $N-1$ or less arrows between the two vertices.
    The open set $O_S$, the set of all mutation classes that restrict to a mutation class in $S$, is the corresponding universal collection.
\end{Rmk}

Having produced new closed and open sets from a mutation-invariant and hereditary property, we now go the other direction.

\begin{Rmk} \label{rmk-isolated-avoiding}
    Let $E_N$ be the set of all $[I_{N+1}]$-avoiding mutation classes, where $I_{N+1}$ is the quiver with no arrows on $N+1$ vertices (the isolated quiver).
    Then each set $E_N$ is closed by Lemma \ref{lem-quiver-avoiding}, and we have an infinite chain of closed subsets
    $$E_1 \subset E_2 \subset E_3 \subset \cdots $$
    with $E_N \neq E_{N+1}$.
    The open set generated by $[I_{N+1}]$ is the corresponding universal collection of $E_N$.
\end{Rmk}

We can then use the previous two families of closed sets to say something about the ascending and descending chain conditions on $\M$.

\begin{Def} \label{def-noetherian}
    A topological space $(X,\mathcal{T})$ is \textit{Noetherian} if it satisfies the \textit{descending chain condition} on closed sets: for any sequence \[ C_1 \supseteq C_2 \supseteq \cdots\]  of closed sets, there exists a positive integer $k$ such that $C_k = C_{k+1} = \cdots.$
\end{Def}

\begin{Prop} \label{cor-acc-dcc}
    As $E_1$ and $C_1$ avoid the same set of quivers, there exists a bi-infinite chain of closed sets of $\M$:
    $$\cdots \subset C_3 \subset C_2 \subset C_1 = E_1 \subset E_2 \subset E_3 \subset \cdots$$
    This shows that $\M$ fails both the ascending chain condition and the descending chain condition on closed sets.
    Hence, the mutation class topology is non-Noetherian.
\end{Prop}

\section{Skew-symmetrizable mutation class topology} \label{sec-iced}

We can extend the construction in Section~\ref{sec-mut} to the set of mutation classes of \textit{skew-symmetrizable matrices}, possibly with frozen indices.
This is the largest mutation class topology that realizes our original quiver mutation class topology as a subspace topology.
In other words, the space described in this section is the largest mutation class Alexandrov space, and all others (including $\M$) are subspaces of this one.

\begin{Def} \label{def-ext-skew-sym}
    If $n \in \mathbb{Z}_{>0}$ and $m \in \mathbb{Z}_{\geq 0}$, then an $(n+m) \times (n+m)$ integer matrix $\widetilde{B} = (b_{ij})$ is a \textit{skew-symmetrizable matrix} if it differs from a skew-symmetric matrix by a positive (integer) rescaling of its rows; in other words, $d_ib_{ij} = - d_jb_{ji}$ for some positive integers $d_1, d_2, \dots, d_{n+m}$. 
    The first $n$ indices of $\widetilde{B}$ are the \textit{mutable} indices, and the remaining $m$ indices of $\widetilde{B}$ are the \textit{frozen} indices. 
\end{Def}

The subset of skew-symmetric integral matrices within the skew-symmetrizable integral matrices corresponds to ice quivers $Q$.
These are quivers with $(n+m)$ vertices, $n$ of which are \textit{mutable} vertices with the remaining $m$ vertices being \textit{frozen}.
For convenience, we label the vertices of $Q$ so that the mutable vertices are labeled $1, 2, \dots, n$, and the frozen vertices are labeled $n+1, \dots, n+m$.

We can also define mutation on skew-symmetrizable matrices, which agrees with the definition of quiver mutation when $\widetilde{B}$ is a skew-symmetric matrix.

\begin{Def}[\cite{fomin_introduction_2021-1}] \label{def-matrix-mut} 
    Let $k \in [n]$ be a mutable index of an $(n+m) \times (n+m)$ skew-symmetrizable matrix $\widetilde{B}$.
    Then the mutation of $\widetilde{B}$ at the index $k$ is the skew-symmetrizable matrix $\mu_k(\widetilde{B}) = (b'_{ij})$ where \[ b'_{ij} = 
    \begin{dcases}
        -b_{ij} & \text{ if } i = k \text{ or } j = k \\
        b_{ij} + b_{ik}[b_{kj}]_+ + [-b_{ik}]_+b_{kj} & \text{ if } i,j \neq k
    \end{dcases}\]
    for $1 \leq i,j \leq (n+m)$, and $[a]_+ = max(a,0)$.
    If $\widetilde{B}$ is a skew-symmetrizable matrix, then the \textit{mutation-class} of $\widetilde{B}$ is the set $[\widetilde{B}]$ consisting of all skew-symmetrizable matrices which can be obtained from $\widetilde{B}$ by a sequence of mutations at mutable indices and a possible relabelling of indices.
\end{Def}

\begin{Rmk}
As for interpretation and connections to existing cluster algebra literature, there is a subset of the mutation classes of skew-symmetrizable integral matrices that has been well studied.
If we consider the subset of our mutation classes of skew-symmetrizable $(n+m) \times (n+m)$ integral matrices $[\widetilde{B}]$ to the mutation classes that contain at least one matrix of the form \[[\widetilde{B}] \ni \widetilde{B}_0 = \begin{pmatrix}
    B_{n\times n} & A_{n\times m} \\
    C_{m\times n} & 0_{m\times m}
\end{pmatrix},\] then these are one of the presentations of \textit{cluster patterns of geometric type} \cite[Proposition 3.22]{nakanishi_cluster_2023}.
Here, the $m$ frozen indices are interpreted as generators in the tropical semifield Trop($u_1, \dots, u_m$), and the lower $m \times n$ matrix of each $\widetilde{B}' \in [\widetilde{B}]$ is the $C$-matrix at that seed.
Therefore, if we are careful, what we are describing is a poset structure and its induced Alexandrov topology on the space of \textit{cluster algebras of geometric type}.
\end{Rmk}

The next definitions are the counterparts of restriction and embedding from Section~\ref{sec-pre}.

\begin{Def} \label{def-restriction}
    Let $\widetilde{B}$ be an $(n+m) \times (n+m)$ skew-symmetrizable matrix.
    If $I$ is a $k$ element subset of the indices, then the skew-symmetrizable matrix $\widetilde{B}_I$ with row set $I$ and column set $I$ is the \textit{restriction} of $\widetilde{B}$ to $I$.
\end{Def}

\begin{Def} \label{def-matrix-embedding}
    We say that a mutation class $[\widetilde{B}]$ \textit{embeds} into a mutation class $[\widetilde{E}]$ and write $[\widetilde{B}] \preceq [\widetilde{E}]$ if there exists skew-symmetrizable matrices $\widetilde{P} \in [\widetilde{B}]$ and $\widetilde{Q} \in [\widetilde{E}]$ such that $\widetilde{P} = \widetilde{Q}_I$ for some set of indices $I$, possibly after a relabelling of indices. 
    In other words, $\widetilde{P}$ is a restriction of $\widetilde{Q}$.
    A property $\mathcal{H}$ is \textit{mutation-invariant and hereditary} if it is preserved under embedding of mutation classes.
\end{Def}

Last but not least, we have an analogous definition for the mutation class topology.

\begin{Def} \label{def-mutation-topology-ext}
    Let $\widetilde{\M}$ be the set of all mutation classes of $(n+m) \times (n+m)$ skew-symmetrizable integral matrices up to permutation of the labeling of indices.
    The \textit{mutation class topology} on $\widetilde{\M}$ is formed by taking the Alexandrov topology generated by $\preceq$, i.e., the closed sets of $\widetilde{\M}$ are the sets $S$ such that $[\widetilde{E}] \in S$ implies that $[\widetilde{B}] \in S$ for all $[\widetilde{B}] \preceq [\widetilde{E}]$.
\end{Def}

The subspace $\M \subset \widetilde{\M}$ of the mutation classes of quivers together with the subspace topology agrees with the topology on $\mathcal{M}$ defined in Section~\ref{sec-mut}.
As before, the closed sets of $\widetilde{\M}$ are in bijective correspondence with mutation-invariant and hereditary properties, and this topology enjoys all of the same properties from Section~\ref{sec-mut} as $\M$.

There are other subspaces of $\widetilde{\M}$ that might be of interest. 
For example, the ``skew-symmetric'' subspace of $\widetilde{\M}$ is the space of ice quivers: quivers with mutable and frozen vertices.
Other options for interesting subspaces are any closed sets of $\widetilde{\M}$: the set of mutation-acyclic skew-symmetric matrices $\MA$ discussed in Section~\ref{sec-pre}, the set of quivers admitting a reddening sequence $\MRed$, the set of $N$-abundant quivers, or the set of skew-symmetrizable matrices with exactly one frozen index are all subspaces of $\widetilde{\M}$ that can be endowed with the subspace topology.
This may bear further exploration.

\section{Open questions} \label{sec-open}

We end this paper with a few possible directions for future research. 
The problems are stated for $\M$, but they could easily be extended to $\widetilde{\M}$; this is analogous to open conjectures in cluster algebra theory which are often first proven for the skew-symmetric case, and then more generally for the skew-symmetrizable case or extended skew-symmetrizable case.

\begin{Open}
    If possible, describe mutation-acyclicity as an intersection of a collection of distinct mutation-invariant and hereditary properties (which are not themselves mutation-acyclicity).
\end{Open}

There is a corresponding definition of acyclic for skew-symmetrizable matrices \cite[Definition 5.11]{nakanishi_cluster_2023} so that we can extend this problem to $\widetilde{\M}$.
If mutation-acyclicity can be decomposed into an intersection of other mutation-invariant and hereditary properties, it may allow for a combinatorial proof that mutation-acyclicity is a hereditary property.

There is also research that can be done involving the closure operator. 

\begin{Open} \label{open-closure}
    If $A$ is the (non-open) set of all mutation classes with a certain mutation-invariant property (not necessarily hereditary), what mutation-invariant and hereditary property corresponds to $\overline{A}$? 
    Can we describe $\overline{A}$ in terms of the original mutation-invariant property?
\end{Open}

\begin{Open} \label{open-banff-louise}
    There are two sets of mutation classes of quivers, Banff and Louise, which are important sets with connections to the study of cluster algebras.
    It is also known that every Louise quiver is Banff and neither of these properties are hereditary.
    What is the closure of Louise? of Banff?
    Are the closures of Louise and Banff distinct closed sets?
\end{Open}

This gives a new technique to deal with OPAC-033 \cite{bucher_banff_2021,ervin_answering_2023} which is concerned with the relationship between Banff and Louise quivers. 
If we understood the closure of these two sets of mutation classes in our topology, we might be able to speak to the equality of the Banff and Louise mutation classes.

We also know that every Banff quiver admits a reddening sequence \cite{bucher_reddening_2020} and that the cluster algebras corresponding to Banff quivers have their upper cluster algebra equal to the cluster algebra ($\mathcal{A} = \mathcal{U}$) \cite{muller_locally_2013}.
If the closure of the Banff quivers is $\MRed$, does that imply that every quiver that admits a reddening sequence has $\mathcal{A} = \mathcal{U}$? 
This could provide an avenue to tackle the conjectured relation between reddening sequences and $\mathcal{A} = \mathcal{U}$ \cite[Conjecture 2]{mills_relationship_2018}.

\begin{Open}
    What is a nice subset whose closure is a given mutation-invariant and hereditary property?
    For example, what is a set whose closure produces all quivers with a reddening sequence and has a minimal number of quivers of any given rank?
    All quivers that are mutation-acyclic?
\end{Open}

These nice subsets (if they exist) would also shed light on the combinatorics of reddening sequences and mutation-acyclicity; this is similar to a question of Bucher and Machacek \cite[Question 3.7]{bucher_reddening_2020}.
For example, suppose a nice subset of quivers whose closure is $\MA$ is a proper subset of $\MA$. In that case, we could direct our study of the combinatorics of mutation-acyclicity to a more focused set of quivers. 
Similarly, if a generating set for $\MRed$ is a proper subset of $\MRed$, we can look to this subset for possible insights into the nature of reddening sequences.

We think it is compelling that so many outstanding questions in cluster algebras can be translated into questions about the mutation class topology. 
It is also worth repeating that the closed sets we have focused on throughout the paper are down sets in a poset.
Therefore it is likely that experts on posets, order ideals, and filter ideals would have something meaningful to add to our construction, and we look forward to seeing those additions and extensions to the theory of this paper.

\subsection*{Acknowledgements}

We would like to thank Kyungyong Lee for his advice and discussions on the topic.
We would also like to thank Ralf Schiffler for his helpful suggestions.
Finally, we would like to thank Ulysses Alvarez and Scott Neville for their comments and critiques of an earlier draft, which significantly improved the paper. 

\bibliography{references}
\bibliographystyle{plain}

\end{document}